\documentclass[12pt,a4paper]{amsart}
\usepackage{color}
\usepackage{amssymb}
\usepackage{graphicx}

\usepackage{psfrag}

\def\Z{\mathbb{Z}}

\def\G{\Gamma}

\def\<{\langle}
\def\>{\rangle}
\def\-{\overline}

\newtheorem{theorem}{Theorem}[section]
\newtheorem{lemma}[theorem]{Lemma}

\theoremstyle{definition}

\numberwithin{equation}{section}

\catcode`\@=11
\def\serieslogo@{\relax}
\def\@setcopyright{\relax}
\catcode`\@=12

\begin{document}

\title{On the recognition of right-angled Artin groups}
%\author{Martin R. Bridson, 28 April 2018} 

% author one information
\author[Martin R Bridson]{Martin R.~Bridson}
\address{Mathematical Institute \\
Andrew Wiles Building\\
Oxford OX2 6GG \\ 
EU} 
\email{bridson@maths.ox.ac.uk}

  
\maketitle

\begin{abstract} 
There does not exist an algorithm that can determine whether or not a group presented by commutators is a right-angled Artin group.
\end{abstract}

\section{Introduction}

In \cite{DW} Day and Wade  introduced an elegant new homology theory for subspace arrangements and related it to recognition problems concerning {\em{right-angled Artin groups}} (RAAGs). 
In setting the context for their work,
they asked if there is an algorithmic procedure for recognizing RAAGs among groups given by   presentations whose only relations are commutators (\cite{DW} Question 1.2) and speculated that the answer was likely to be no. The purpose of this note is to confirm this speculation.

\begin{theorem}\label{thm} There does not exist an algorithm that can determine whether or not a group presented by commutators is a RAAG.

In more detail, there is no algorithm that, given $22$ words $u_i$ in the free group $F(a_1,a_2,a_3,a_4)$ can determine whether or not the group with presentation
$$\<a_1,a_2,a_3,a_4, t \mid [a_1,a_3], [a_1,a_4], [a_2,a_3], [a_2,a_4], [t,u_1],\dots,[t,u_{22}]\>$$
  is a RAAG. Nor is there an algorithm that can determine whether or not such a group is commensurable with a RAAG or quasi-isometric to a RAAG.
\end{theorem}

\section{Fibre products and triviality for 2-generator groups} 

In the aftermath of the construction by Novikov \cite{nov} and Boone \cite{boone}
of finitely presented groups with unsolvable word problem, 
%Adian \cite{Ad} and Rabin \cite{Rab} introduced a technique for deducing that 
many other decisions problems for groups were proved to be
unsolvable through subtle work by many authors. We shall
appeal to two results that come from the work of C.F. Miller III.

Let $K$ be a group given by a presentation with $n$ generators and
$m$ relations. Following Miller, one can associate to each word $w$
in the generators of $K$ a presentation with $2$ generators and $n+m+3$
relations -- this is derived from the presentation
in Lemma 3.6 of \cite{cfm2} by making Tietze moves to remove  unnecessary generators.
The group given by this presentation is trivial if $w=1$ in $K$, but it contains 
$K$ if $w\neq 1$. 

The most concise finite presentation that is known for a group with 
unsolvable word problem is the one constructed by Borisov \cite{boris} half a 
century ago -- it has $5$ generators and
$12$ relations. By applying Miller's construction to words in the generators of
Borisov's example, we obtain a recursive sequence of 
2-generator presentations $\mathcal{P}_n = \<a_1,a_2\mid S_n\>$ with $|S_n|=20$ such that 
the group presented is either trivial or else has an unsolvable word problem, and there
is no algorithm that can determine the set of integers $n$ for which each alternative holds. 

We exploit these examples  in the manner of  Mihailova \cite{mihailova} and
Miller  (\cite{cfm} p.39).

\begin{lemma} \label{lemma}
Let $F$ be a free group of rank $2$ with generators $\{a_1,a_2\}$. If $k\ge 20$,
then there exists  a recursive sequence
$(S_n)$ of subsets of $F$, each of  cardinality $k$, such that there is no algorithm to
determine whether or not $F\times F$ is generated by $U_n=\{(a_1,a_1), (a_2,a_2),(s,1): s\in S_n\}$. 
Moreover, if $\<U_n\>$ is a proper subgroup then 
$H_2(\<U_n\>,\Z)$ is not finitely generated and there is
no algorithm to determine which words in the generators of $F\times F$ 
determine elements of $\<U_n\>$.
\end{lemma}

\begin{proof} Associated to any 2-generator finite presentation
$\mathcal{P}=\<a_1,a_2\mid r_1,\dots,r_M\>$
one has the
fibre product
$P <F\times F$ consisting of pairs $(u,v)$ such that $u=v$ in the group 
$G(\mathcal{P})$ presented by $\mathcal{P}$. It is easy to check that
$P$ is generated by $U=\{(a_1,a_1), (a_2,a_2), (r_1,1),\dots, (r_M,1)\}$. If  $G(\mathcal{P})$ is trivial, $P=F\times F$.
But if $G(\mathcal{P})$ is infinite, Baumslag and Roseblade's
Theorem A \cite{BR} shows that $H_2(P,\Z)$ is not finitely generated. 
Moreover, if the word problem is unsolvable
in $G(\mathcal{P})$, the membership problem of $P<F\times F$ is unsolvable,
because deciding if $(w,1)\in P$ is equivalent to deciding if $w=1$ in $G(\mathcal{P})$.
Consideration of the presentations $\mathcal{P}_n = \<a_1,a_2\mid S_n\>$
from the discussion preceding the lemma completes the proof.
\end{proof}

We shall apply the following lemma with $D=F\times F$ and $C=\<U_n\>$.

\begin{lemma}\label{l:mv} For any HNN extension of the form $\G=D\ast_C$,
if $H_2(D,\Z)$ is finitely generated but $H_{2}(C,\Z)$
is not, then $H_{3}(\G,\Z)$ is not finitely generated.
\end{lemma}

\begin{proof} The Mayer-Vietoris sequence for the HNN extension contains the exact
sequence
$$  H_{3}(\G,\Z) \to H_{2}(C,\Z) \to H_{2}(D,\Z).
$$ 
\end{proof}

\section{Proof of Theorem 1.1}  Given  20 words  $S=\{r_1,\dots,r_{20}\}$ in the free group $F=F(a_1,a_2)$, we denote by $\G({S})$ the group with
generators $a_1,a_2,a_3,a_4, t$ and relations
$$
[a_1,a_3], [a_1,a_4], [a_2,a_3], [a_2,a_4], [t,a_1a_3], [t,a_2a_4], [t,r_1],\dots,[t,r_{20}].
$$
Note that this presentation is of the type described in Theorem \ref{thm}. The group $\G({S})$ is an HNN extension of $F\times F$ with a stable letter $t$ that commutes with
the fibre product $P<F\times F$ associated to the presentation $\mathcal{P}=\<a,b\mid S\>$.

As in the proof of Lemma \ref{lemma}, we have a dichotomy:
if $G(\mathcal{P})=1$ then $\G({S}) = F\times F \times \mathbb{Z}$ is a RAAG; but
if $G(\mathcal{P})$ is infinite then $\G({S})$ is not
a RAAG, because RAAGs have finite classifying spaces \cite{KR} whereas
$H_3(\G({S}), \Z)$ is not finitely generated, by Lemma \ref{l:mv}, because
 $H_2(P,\Z)$ is not finitely generated.

Lemma \ref{lemma} tells us that there is no algorithm to determine which of the possibilities in this dichotomy 
holds. Moreover, by choosing sets $(S_n)$ as in the lemma, we can
arrange that when $\G(S_n)$ is not equal to  $F\times F\times\mathbb{Z}$, it will
have an unsolvable word problem: given a word $w$ in the generators $a_1,a_2$, Britton's Lemma implies that $[t,w]=1$ in $\G(S_n)$ if and only if $w=1$ in $G(\mathcal{P}_n)$,
and this is undecidable.

Since being of type ${\rm{FP}}_3$ and having
a solvable word problem are both invariants of commensurability and quasi-isometry,
there is no algorithm that can determine 
whether $\G({S_n})$ is commensurable with or
quasi-isometric to a RAAG.
$\square$

\end{document}